\documentclass{amsart}
\usepackage{amsfonts, amsmath,  amssymb, graphicx, color}
\newtheorem{theorem}{Theorem}

\newtheorem{conjecture}{Conjecture}
\newtheorem{corollary}{Corollary}

\newtheorem{lemma}{Lemma}

\theoremstyle{remark}
\newtheorem{remark}{Remark}
\newtheorem{construction}{Construction}
\theoremstyle{definition}
\newtheorem{definition}{Definition}
\numberwithin{equation}{section}
\newcommand{\Po}{{\mathcal P}}
\newcommand{\I}{{\mathcal I}}
\newcommand{\Li}{{\mathcal L}}
\newcommand{\Cl}{{\mathcal C}}

\newcommand{\Q}{{\mathcal Q}}
\newcommand{\G}{{\mathcal G}}

\begin{document}
\title[Embedding cycles in finite planes]{Embedding cycles in finite planes}
\author{Felix Lazebnik}
\address{Department of Mathematical Sciences\\
University of Delaware\\
Newark, DE 19716.}
\email{lazebnik@math.udel.edu}
\author{Keith E. Mellinger}
\address{Department of Mathematics\\
University of Mary Washington\\
Fredericksburg, VA  22401.}
\email{kmelling@umw.edu}
\author{Oscar Vega}
\address{Department of Mathematics\\
California State University, Fresno \\ Fresno, CA 93740.}
\email{ovega@csufresno.edu}
%\urladdr{http://zimmer.csufresno.edu/$\sim$ovega}
\date{}

\subjclass[2000]{Primary 05, 51; Secondary 20} 

\begin{abstract}
We define and study embeddings of cycles in finite affine and projective planes. We show that for all $k$, $3\le k\le q^2$,  a $k$-cycle can be embedded in any affine plane of order $q$. We also prove a similar result for finite projective planes: for all $k$, $3\le k\le q^2+q+1$,  a $k$-cycle can be embedded in any projective plane of order $q$.
\end{abstract}

\maketitle

%%%%%%%%%%%%%%%%%%%%%%%%%%%%%%%%%%%%%%%%%%%%%%%%%%%%%%%%%%%
\section{Introduction}

Our work concerns substructures in finite affine and projective planes. In order to explain the questions we consider, we will need the following definitions and notations.

Any graph-theoretic notion not defined here may be found in Bollob\' as \cite{Bol98}. All of our graphs are finite, simple and undirected. If $G = (V,E)=(V(G), E(G))$ is a graph, then the {\it order} of $G$ is $v(G)=|V|$, the number of vertices of $G$, and the {\it size} of $G$ is $e(G)=|E|$, the number of edges in $G$. Each edge of $G$ is thought as a 2-subset of $V$. An edge $\{x,y\}$ will be denoted by $xy$ or $yx$. A vertex $v$ is {\it incident} with an edge $e$ if $v\in e$. We say that a graph $G'=(V',E')$ is a {\it subgraph of $G$},  and denote it by $G'\subset G$, if  $V'\subset V$ and $E'\subset E$. If $G' \subset G$, we will also say that $G$ {\it contains} $G'$. For a vertex $v \in V$, $N(v) = N_G(v)= \{u \in V: uv \in E\}$ denotes the {\it neighborhood} of $v$, and $deg(v) = deg_G(v) = |N(v)|$, the {\it degree} of $v$. If the degrees of all vertices of $G$ are equal to $d$, then $G$ is called $d$-{\it regular}. For a graph $F$, we say that $G$ is {\it $F$-free} if $G$ contains no subgraph isomorphic to $F$.

For $k\ge 2$, any graph isomorphic to the graph with a vertex-set $\{x_1,\ldots, x_k\}$ and an edge-set $\{x_1x_2, x_2x_3,\ldots, x_{k-1}x_k\}$ is called an $x_1x_k$-path, or a {\it $k$-path},  and we denote it by  ${\mathcal P}_k$.  The length of a path is its number of edges. For $k\ge 3$, the graph with a vertex-set $\{x_1,\ldots, x_k\}$ and edge-set $\{x_1x_2, x_2x_3, \ldots , x_{k-1}x_k, x_kx_1\}$ is called a {\it $k$-cycle},  and it is often denoted by $\Cl$ or $\Cl_k$. Any subgraph of $G$ isomorphic to a $k$-cycle is called a {\it $k$-cycle in $G$}.  The {\it girth} of a graph $G$ containing cycles, denoted by $g=g(G)$, is the length of a shortest cycle in $G$. Let $V(G)=A\cup B$ be a partition of $V(G)$,  and let every edge of $G$ have one endpoint in $A$ and another in $B$. Then $G$ is called {\it bipartite} and we denote it by $G(A,B;E)$. If $|A|=m$ and $|B|=n$, then we refer to $G$ as an $(m,n)$-bipartite graph.

All notions of incidence geometry not defined here may be found in \cite{Bu95}. A {\it partial plane} $\pi = ({\Po}, {\Li} ; \I)$ is an incidence structure with a set of points $\Po$, a set of lines $\Li$, and a symmetric binary relation of incidence $I\subseteq ({\Po}\times {\Li})\,\cup \,({\Li}\times {\Po})$ such that any two distinct points are on at most one line, and every line contains at least two points (note that we have used $\Po$ for two different object as of now: to denote a path and to denote the points on a partial plane. The usage of this symbol should be clear from the context).  The definition implies that any two lines share at most one point. We will often identify lines with the sets of points on them. We say that a partial plane $\pi ' = ({\Po}', {\Li}' ; \I')$ is a {\it subplane} of $\pi$,  denoted $\pi' \subset \pi$,  if  ${\Po}'\subset {\Po}, {\Li}'\subset{\Li}$,  and $\I'\subset \I$. If there is a line containing two distinct points $X$ and $Y$, we denote it by $XY$ or $YX$.  For $k\ge 3$,  we define a $k$-gon as a partial plane with $k$ distinct points $\{P_1,P_2, \ldots P_k\}$, with $k$ distinct lines $\{P_1P_{2}, P_2P_{3}, \ldots, P_{k-1}P_k, P_kP_1\}$, and with point and line being incident if and only if the point is on the line. A subplane of $\pi$ isomorphic to a $k$-gon is called a {\it $k$-gon in $\pi$}.  The {\it Levi graph} of a partial plane $\pi$ is its point-line bipartite incidence graph $Levi(\pi) = Levi({\Po},{\Li}; E)$, where $Pl\in E$ if and only if point $P$ is on line $l$. The Levi graph of any partial plane is $4$-cycle-free. Clearly, there exists a bijection between the set of all $k$-gons in $\pi$ and the set of $2k$-cycles in  $Levi(\pi)$.

A {\it projective plane of order $q\ge 2$}, denoted $\pi_q$, is a partial plane with every point on exactly $q+1$ lines, every line containing exactly $q+1$ points,  and having four points such that no three of them are collinear. It is easy to argue that $\pi_q$ contains $ q^2 + q + 1$ points and $q^2 + q + 1$ lines.  Let $n_q = q^2 + q + 1$.  It is easy to show that a partial plane is a projective plane of order $q$ if and only if its Levi graph is a $(q+1)$-regular graph of girth $6$ and diameter $3$. Projective planes $\pi_q$ are known to exist only when the order $q$ is a prime power. If $q\ge 9$ is a prime power but not a prime, there are always non-isomorphic planes of order $q$, and their number grows fast with $q$.  Let $PG(2,q)$ denote the {\it classical} projective plane of prime power order $q$ which can be modeled as follows:  points of $PG(2,q)$ are 1-dimensional subspaces in the 3-dimensional vector space over the finite field of $q$ elements, lines of $PG(2,q)$ are 2-dimensional subspaces of the vector space,  and a point is incident to a line if it is a subspace of it.

Removing a line from a projective plane, and removing its points from all other lines,  yields a partial plane known as an {\it affine plane}.  The line removed is often referred to as the {\it  line at infinity}, and it is denoted by $l_\infty$.  Conversely,  a projective plane of order $q$ can be obtained from an affine plane of order $q$ (i.e. having  $q+1$ lines through each point) by adding a line at infinity to it, which can be thought of as a set of $q+1$ new points, called {\it points at infinity}, which is in bijective correspondence with the set of parallel classes (also called the set of all slopes) of lines in the affine plane. %More details about this construction may be found in \cite{H&P}. \\
We will use $\pi_q$ to denote a projective plane of order $q$, and $\alpha_q$ for affine planes of order $q$.

\medskip

The following problem,  stated in terms of set systems, appears in Erd\H os ~\cite{Erd79}:

\medskip

\noindent {\bf Problem 1.} \label{Problem1} {\it Is  every finite partial linear space embedded in a finite projective plane?}

\medskip

It is possible that the question was asked before,  as it was well known that every partial linear space embeds in some infinite projective plane, by a process of free closure due to Hall \cite{Hall43}.  For recent results related to the question, see Moorhouse and Williford \cite{MoorWil2009}.   Rephrased in terms of graphs,  Problem 1 is the following:

\medskip

\noindent {\bf Problem 1$^*$.} \label{Problem1'} {\it Is every finite bipartite graph without 4-cycles a subgraph of the Levi graph of a finite projective plane?}

\medskip

Thinking about cycles in Levi graphs of projective planes,  we introduced  the following notion of embedding of a graph into a partial plane,  and found it  useful.  Let $G$ be a graph and let $\pi = ({\Po}, {\Li}; \I)$ be a partial plane. Let
\[
f: V(G)\cup E(G) \to {\Po}\cup {\Li}
\]
be an injective map such that $f(V(G))\subset {\Po}$,  $f(E(G))\subset \Li$,  and for every vertex  $x$ and edge $e$ of $G$, their incidence in $G$ implies the incidence of  point $f(x)$ and line $f(e)$ in $\pi$.  We call such a map $f$ an {\it embedding of $G$ in $\pi$},  and if it exists  we say that {\it $G$ embeds in $\pi$}  and write $G\hookrightarrow \pi$.  If $G\hookrightarrow \pi$,  then adjacent vertices of $G$ are mapped to collinear points of $\pi$. Note that  if $G\hookrightarrow \pi_q$,  then $v(G)\le n_q$, $e(G)\le n_q$,  and $deg_G(x)\le q+1$ for all $x\in V(G)$.

A cycle containing all vertices of a graph is called a {\it hamiltonian cycle} of the graph, and if such exists, the graph is called a {\it hamiltonian} graph. Similarly,  if $\pi_q$ contains an $n_q$-gon, we call it  {\it hamiltonian}. A graph $G$ containing $k$-cycles of all possible lengths, $3\le k\le v(G)$, is called {\it pancyclic}. Similarly,  we say that $\pi _q$ is {\it pangonal}, if it contains $k$-gons for all $3\le k\le n_q$.  The latter is equivalent to $Levi (\pi_q)$ containing all $2k$-cycles for $3\le k\le n_q$.   Clearly,  if $G\hookrightarrow \pi_q$, a $k$-cycle in $G$ corresponds to a $k$-gon in $\pi_q$, which, in turn, corresponds to a $2k$-cycle in $Levi(\pi_q)$.  From now on we choose to be less pedantic, and will  feel free to use graph theoretic and geometric  terms interchangeably.  For example, we will say `point' for a vertex of a graph, `vertex' for a point of a partial plane, and we will speak about `path' and `cycle' in a plane, etc.

Determining whether a graph is hamiltonian,  or, more generally,  understanding what cycles it contains, is one of the central problems in graph theory,  and  it has been a subject of active research for many years. The existence of hamiltonian cycles in $\pi_q$ (or $Levi (\pi_q)$), or its pancyclicity, was addressed by several researchers. The presence of $k$-gons of some small lengths in $\pi_q$ is easy to establish.  In \cite{LMV09},  the authors presented explicit formuli for the numbers of distinct $k$-gons in every projective plane of order $q$ for $k=3, 4, 5, 6$.  Very recently, and in a very impressive way,  Voropaev \cite{Vor12} extended this list  to $k= 7, 8, 9, 10$. The existence of very special hamiltonian cycles in $PG(2,q)$ is a celebrated result of Singer \cite{Singer38}. These cycles are often referred to as the {\it Singer cycles} in $PG(2,q)$.  For $q=p$ (prime)  Schmeichel \cite{Schm89} showed by explicit constructions that $PG(2,p)$ is pancyclic, and  that  the hamiltonian cycles he constructed were different from Singer cycles.  DeMarco and Lazebnik  \cite{DL08} constructed a hamiltonian cycle in a Hall plane of order $p^2$.  Most of the known  sufficient conditions for the existence of hamiltonian cycles in graphs are effective for rather dense graphs: graphs of order $n$ and size greater that $cn^2$ for some positive constant $c$ (see a survey by Gould \cite{Gould03}).  Levi graphs of projective planes are much sparser; being $(q+1)$-regular, their size is $ (1/(2\sqrt{2}) +o(1))n^{3/2}$ for $n\to \infty$,  and that is why most techniques of proving  hamiltonicity of graphs do not apply to them.  For the same reason, upper bounds on the Tur\'an number of a $2k$-cycle,  see, e.g., Pikhurko \cite{Pik12} and and references therein, are not effective for proving the existence of $2k$-cycles in Levi graphs of projective planes for most values of $k$ (as $k$ may depend on $q$).

\medskip

A new approach for establishing hamiltonicity and the existence of shorter cycles came from probabilistic techniques and studies of cycles in random and pseudo-random graphs (we omit the definition).  See, e.g., Thomasson \cite{Thom87}, Chung, Graham and Wilson \cite{CGW89}, Frieze \cite{Fri00}, and Frieze and Krivelevich \cite{FriKriv02}.

In \cite{KS03}, Krivelevich and Sudakov explored relations between pseudo-randomness and hamiltonicity in regular non-bipartite graphs.  Some other results related to hamiltonicity and pancyclicity appeared in recent publications by Keevash and Sudakov \cite{KeeSud10}, Krivelevich, Lee and Sudakov \cite{KriLeeSud10}, and Lee and Sudakov \cite{LeeSud12}.

It is likely that proofs  in these papers can be modified to give results for (bipartite) Levi graphs of projective planes,  but the requirement on the order of the graph to be sufficiently large (as is the case in the aforementioned papers) will remain.  In this paper  we establish the pancyclicity of $\pi_q$ and $\alpha_q$, for all $q$, and our proof is constructive. \\

Our main results follow.

\begin{theorem} \label{thmlongcycles}
Let $\alpha_q$ be an affine plane of order $q\ge 2$.  Then $C_k \hookrightarrow \alpha_q$  for all $k$, $3\le k\le q^2$.
\end{theorem}

\begin{theorem}\label{main}
Let $\pi_q$ be a projective plane of order $q\ge 2$, and $n_q=q^2 + q + 1$.   Then $C_k \hookrightarrow \pi_q$  for all $k$, $3\le k\le n_q$.
\end{theorem}

We now proceed to give a construction for paths and cycles in \emph{any} finite affine or projective plane. We start with a remark that will be very useful later on.

\begin{remark}\label{remcombpaths}
Let $P_1 \rightarrow P_2 \rightarrow  \cdots \rightarrow  P_k$ and $Q_1\rightarrow Q_2\rightarrow  \cdots \rightarrow  Q_n$ be two disjoint (in terms of points and lines) paths embedded in $\pi_q$ or $\alpha_q$.  Then, if the line $\ell = P_kQ_n$ has not been used in these embeddings, we can create the following embedding for a path on $n+k$ vertices:
\[
P_1\rightarrow P_2\rightarrow  \cdots \rightarrow  P_k   \xrightarrow{\ell}  Q_n \rightarrow Q_{n-1} \rightarrow \cdots  Q_{2}  \rightarrow  Q_1.
\]
Here, the symbol $P_k \xrightarrow{\ell} Q_n$ indicates that the line $\ell$ joins the points $P_k$ and $Q_n$. Moreover, if the line $m = Q_1P_1$ is still available, then we get a cycle of length $k+n$ embedded in $\pi_q$ (or $\alpha_q$).\\
\end{remark}

Our main technique in the next two sections will be to construct paths that can be combined using Remark \ref{remcombpaths} to create cycles of any length.

\bigskip

%%%%%%%%%%%%%%%%%%%%%%%%%%%%%%%%%%%%%%%%%%%%%%%%%%%%%%%%%%%
\section{Cycles in Affine Planes}\label{sec:CyclesAffinePlanes}

Let $\alpha_q$ be an affine plane of order $q$, and let $O$ be any point of the plane. We label the $q+1$ lines through $O$ by $l_0, l_1, \ldots , l_{q}$. For any given point $Q\in \alpha_q$, we use $l_{i}+Q$ to denote the line parallel to $l_i$ that passes through $Q$.  Let  $ a  \mod  q+1$ denote the remainder of the division of $a$ by $q+1$.

Pick any point $P_0$ on $l_0$, different from $O$. Let $P_1$ be the point of intersection of $l_{2}+P_0$ and $l_1$. Let $P_2$ be the point of intersection of $l_3+P_1$ and $l_2$, etc. In general, let $P_i$ be the point of intersection of $l_{i+1 \bmod q+1}+P_{i-1}$ and $l_i$, for all $i=1,2,\ldots , q$.  Since $O\neq P_i \in l_i$, for all $i=1, 2, \ldots, q$, then all these points are distinct. Similarly, the lines $P_{i-1} P_{i}$ are in different parallel classes, for all $i=1, 2, \ldots, q$. It follows that by joining the points $P_{i-1}$ and $P_{i}$, for all $i=1, 2, \ldots, q$, we obtain a path on $q+1$ vertices. Denote this path by $\mathcal{P}_{P_0}$. \\

\newpage

\begin{figure}[htbp]
\begin{center}
\includegraphics[height=1.5in]{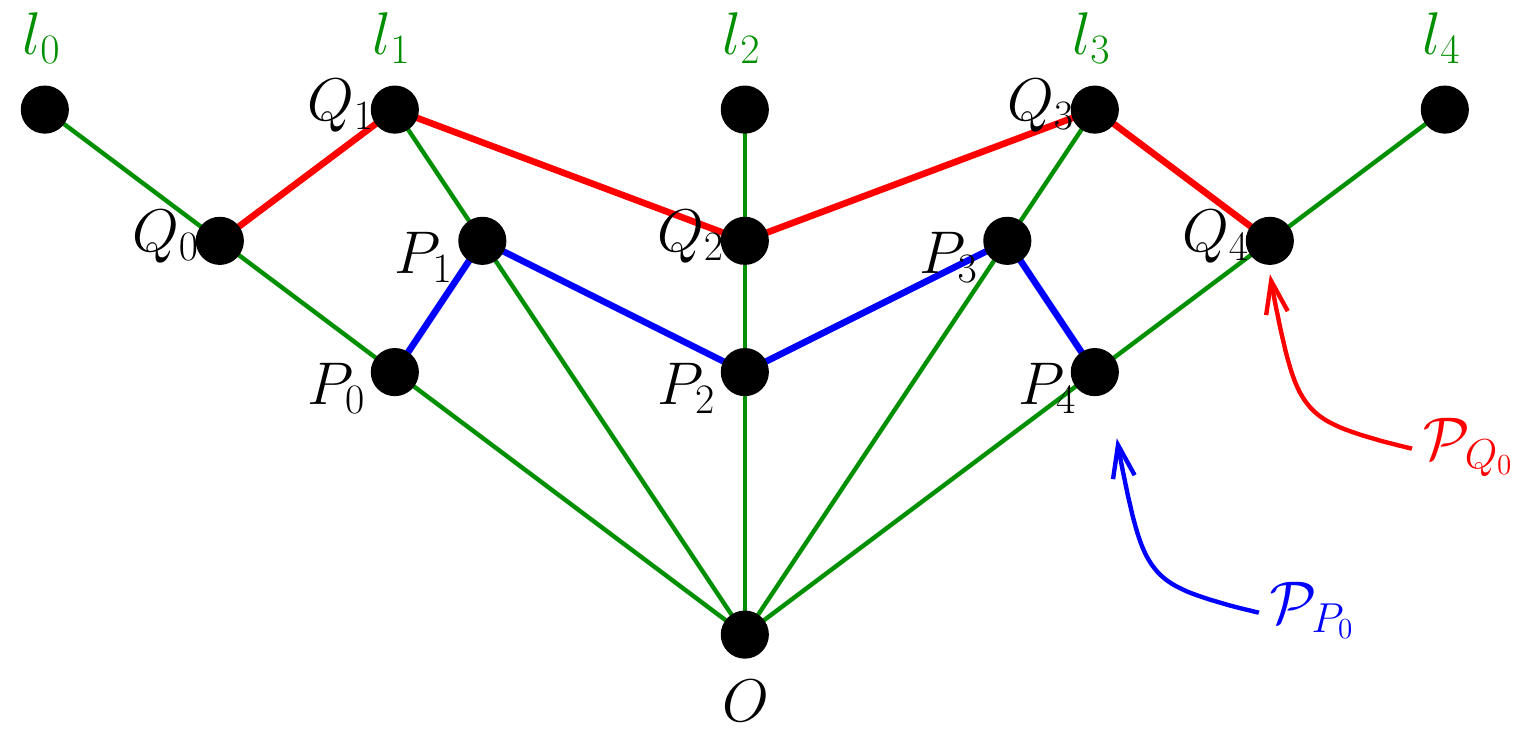} \\
\caption{Two vertex/edge disjoint paths, $\mathcal{P}_{P_0}$ and $\mathcal{P}_{Q_0}$, for $q=4$.}
\end{center}
\end{figure}

\begin{lemma}\label{lemdisjoint}
Let $P_0 \neq Q_0 \in l_0$.  Then the paths $\mathcal{P}_{P_0}$ and $\mathcal{P}_{Q_0}$ share neither points nor lines.
\end{lemma}

\begin{proof}
Let
\[
\mathcal{P}_{P_0}: \ P_0 \rightarrow P_1\rightarrow  \cdots \rightarrow  P_q \hspace{1in} \mathcal{P}_{Q_0}: \  Q_0 \rightarrow Q_{1} \rightarrow \cdots  Q_q
\]
Clearly $P_i \neq Q_j$, for $i\neq j$. We also know that $P_0 \neq Q_0$. So, assume that $P_i=Q_i$, for some $i=1, \ldots , q$, so that $P_j\neq Q_j$, for all $0\leq j<i$. It follows that
\[
\left( l_{i+1 \bmod q+1}+P_{i-1} \right) \cap  l_i = P_i = Q_i =\left( l_{i+1 \bmod q+1}+Q_{i-1}  \right)\cap  l_i
\]
which forces $l_{i+1 \bmod q+1}+P_{i-1} = l_{i+1 \bmod q+1}+Q_{i-1} $, and thus $P_{i-1}=Q_{i-1}$, a contradiction.

Finally, it is easy to see that if $\mathcal{P}_{P_0}$ and $\mathcal{P}_{Q_0}$ shared a line, then they would also share a point.
\end{proof}

\begin{lemma}\label{lemscycles}
We can partition the points of $\alpha_q \setminus \{O\}$ into $s$ cycles, $\Cl_1, \ldots , \Cl_s$, where the length of $C_i$ is $t_i(q+1)$  for some integer $t_i$, $1\le  i  \le   s  \le q-1$, $1\le  t_1 \le   \ldots  \le t_s$, and $t_1 + \cdots + t_s = q-1$.
\end{lemma}

\begin{proof}
If we label the points on $l_0 \setminus\{O\}$ by $x_1, x_2, \cdots, x_{q-1}$, by Lemma \ref{lemdisjoint}, $\mathcal{P}_{x_1}$, $\mathcal{P}_{x_2}, \ldots$, $\mathcal{P}_{x_{q-1}}$ yields a partition of the points of $\alpha_q \setminus \{O\}$ into $q-1$ disjoint paths each having $q+1$ vertices. If $q\ge 3$, then we have at least two such paths,  and  we may want to connect them to create longer paths and/or cycles.

Note that in the paths $\Po_{x_i}$, no line parallel to $l_1$ has been used. Now, if we consider a path $\mathcal{P}_{P_0}$,  then the line $l_{1}+ P_{q}$ intersects $l_0$ at a point $Q$, which can never be equal to $O$, otherwise $l_{1}+ P_{q}=l_1$, and thus $P_q\in l_1\cap l_q = \{O\}$, a contradiction. This point $Q$ is uniquely determined by $P_0$ (and the way we do this construction, of course). If $Q=P_0$, then we get a $(q+1)$-cycle. If $Q\neq P_0$, then we re-label $Q=Q_0$ and consider the path $\mathcal{P}_{Q_0}$. This will give us a path with $2(q+1)$ vertices, namely
\[
P_0 \rightarrow P_1\rightarrow  \cdots \rightarrow  P_q \rightarrow Q_0 \rightarrow Q_{1} \rightarrow \cdots  Q_q.
\]
We then proceed to find $R_0: = \left( l_{1}+ Q_{q} \right) \cap l_0$. If $R_0=P_0$ we get a cycle of length $2(q+1)$. If $R_0 =Q_0$, then we get that $Q_0$ is on two lines that are parallel to $l_1$, namely $l_{1}+ Q_{q}$ and $l_{1}+ P_{q}$. This forces $P_q$ and $Q_q$ to coincide, but this is impossible because of Lemma \ref{lemdisjoint}. It follows that we either get a cycle of length $2(q+1)$ or we can keep extending this path using $\mathcal{P}_{R_0}$. Since $l_0$ contains finitely many points this process must end. Moreover, it is impossible to `close' this cycle at any point that is not $P_0$, as this would yield the same contradiction we obtained above when we assumed $R_0 =Q_0$. Hence,  by combining paths we can construct cycles of length $t(q+1)$, for some positive integer $t$, these are  the cycles $\Cl_i$ we wanted to find.
\end{proof}

\medskip

In order to prove Theorem  \ref{thmlongcycles} we will need to construct paths out of the cycles $\Cl_1, \Cl_2,\ldots, \Cl_{s}$. Firstly, we define terms and set notation that will be necessary for the rest of this article.

\begin{definition}
For every $i=1,\ldots, s$, let $P_{i,i-1}$ be an arbitrary point on $l_{i-1}\cap \Cl_i$ (note that there are $t_i$ such points), and let $P_{i,i}$ be its neighbor on $l_i$.

We construct two different types of paths in $\Cl_{i}$: all of them start at $P_{i,i-1}$ and
\begin{enumerate}
\item the next vertex is $P_{i,i}$. The other vertices in the path are easily determined from these first two, or
\item the next vertex is the neighbor of $P_{i,i-1}$ in $\Cl_{i}$ that is on the line $l_{i-2 \mod q+1}$. The other vertices in the path are easily determined from these first two.
\end{enumerate}
We will say that the first path is a \emph{positive} path, and that the second is a \emph{negative} path.
\end{definition}

\begin{lemma}\label{lemcaset_1}
$k$-cycles can be embedded in $\alpha_q$, for all $3\leq k \leq t_1(q+1)$.
\end{lemma}

\begin{proof}
If $q=2, 3$ the result is immediate. We assume $q\ge 4$ for the rest of this proof.

The cycle $\Cl_1$ is of length $t_1(q+1)$,  and so we only need to construct  $k$-cycles with $3  \le  k <   t_1(q+1)$. \\
If   $k\equiv  1 \pmod{q+1}$,  then,  since $k\ge 3$,  we consider a  positive $k$-path in $\Cl_1$  starting at $P_{1,0}$. As $k\equiv  1 \pmod{q+1}$,   this path ends at some $ Q_0 \in l_0$, and $Q_0 \ne P_0$. Connect $P_0$ to $Q_{0}$ using $l_0$ to get a $k$-cycle. \\
If $k \equiv 2 \pmod{q+1}$ and $2< k<t_1(q+1)$, then $t_1>1$. Consider a positive $(k-2)$-path $\Po$ in $\Cl_1$ starting at $P_{1,0}$. This path ends at a point $P_q\in l_q$. Since $k<t_1(q+1)$ then there is a $2$-path in $\Cl_1$, disjoint from $\Po$, of the form $Q_{q-1} \rightarrow Q_q$ with $Q_i\in l_i$, for $i=q-1,q$. Consider the following $k$ cycle
\[
O \xrightarrow{l_0}  \underbrace{P_0 \rightarrow \cdots \rightarrow P_{q-1}}_{in \ \Po}  \xrightarrow{l_{q-1}}  Q_{q-1}\rightarrow Q_q \xrightarrow{l_q}  O,
\]
where $P_{q-1}\in l_{q-1}$ was the neighbor of $P_q$ in $\Po$.  \\
If   $k \not\equiv  1,2  \pmod{q+1}$,  then,  since $3\le k\le  t_1(q+1)$,  take a positive $(k-1)$-path in $\Cl_1$ starting at $P_{1,0}$. This path will end on a point $P_{k-2}\in l_{k-2}$. Connect $P_0$ and $P_{k-2}$ to $O$ using $l_0$ and $l_{k-2}$, respectively, to get a $k$-cycle.
\end{proof}

\vspace{.1in}

We now focus on the construction of $k$-cycles for $k > t_1(q+1)$. In order to do that we will use the following construction.

\vspace{.1in}

\begin{construction}\label{constrP_m}
Let $\lambda_m =t_1+t_2+\cdots + t_m$. We will construct a  $\lambda _m (q+1)$-path $\mathcal{P}_m$ out of the cycles $\Cl_1, \Cl_2,\ldots, \Cl_{m}$, where $2\leq m \leq s$ (recall that $s\leq q-1$). \\
For each $i=1,\ldots, m-1$, we connect $\Cl_i$ with $\Cl_{i+1}$ by joining $P_{i,i}$ with $P_{i+1,i}$ using $l_i$. Then, for all $i=1,\ldots, m$, we take the $P_{i,i-1}P_{i,i}$ path in $\Cl_i$ having $t_i(q+1)$ vertices, and construct the following path
\[
\underbrace{P_{1,0} \rightarrow \cdots \rightarrow P_{1,1}}_{in \ \Cl_1}  \xrightarrow{l_1} \underbrace{P_{2,1} \rightarrow \cdots \rightarrow P_{2,2}}_{in \ \Cl_2}  \xrightarrow{l_2} \cdots  \xrightarrow{l_{m-1}} \underbrace{P_{m,m-1} \rightarrow \cdots \rightarrow P_{m,m}}_{in \ \Cl_m}
\]
Since no vertices were eliminated or added, and all new lines are distinct and through $O$ (none used in the
construction of the $\Cl_i$'s), this construction yields a  $P_{1,0}P_{m,m}$-path with  $\lambda _m (q+1)$
vertices.
\begin{figure}[htbp]
\begin{center}
\includegraphics[height=2.55in]{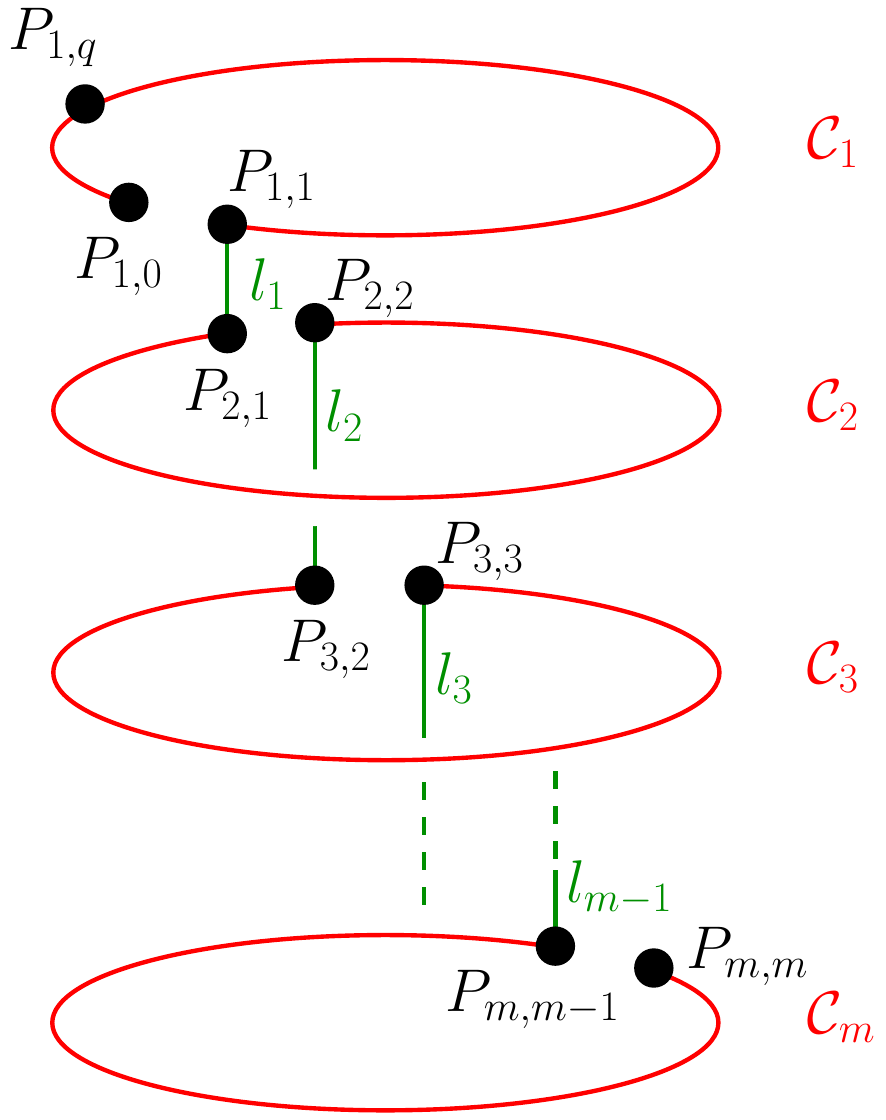}  \\
\caption{Construction of $\mathcal{P}_m$}
\end{center}
\end{figure}

\noindent Note that $O$ has not been used in the construction of $\mathcal{P}_m$, and that neither have the lines $l_m, \ldots, l_q$, and $l_0$. \\
Finally, we will denote the neighbor of $P_{1,0}$ in $\mathcal{P}_m$, which is a point on $l_q$, by $P_{1,q}$.
\end{construction}

\medskip

We now prove Theorem \ref{thmlongcycles}.

\begin{proof}[Proof of Theorem \ref{thmlongcycles}]
In this proof we follow the notation introduced in Construction \ref{constrP_m}.

If $q=2$,  the existence of 3-  and 4-cycles is obvious.  If $q=3$, pancyclicity can be easily verified. In what follows  we assume $q\ge 4$,  though most arguments hold for $q\ge 3$.

We want to embed all possible $k$-cycles in $\alpha_q$ that have not been already discussed in Lemma \ref{lemcaset_1}. For any given $k$, we  write it as either $k=\lambda_s(q+1)$, $k=\lambda_s(q+1)+1 = q^2$, or $k = \lambda_m(q+1)+r$, for some $m=1, \ldots, s-1$ and $0\leq r<t_{m+1}(q+1)$. Note that the case $m=0$ was taken care of in Lemma \ref{lemcaset_1}.

Firstly, we can join $P_{1,0}$ and $P_{s,s}$ with $O$, using the lines $l_0$ and $l_s$ respectively, to obtain a cycle of length $\lambda_s(q+1)+1$. Note that this grants hamiltonicity. Moreover, if we cut $\mathcal{P}_s$ short one vertex, and thus we ask it to end at $P_{1,q}$, then joining the endpoints of this new path to $O$ yields a $\lambda_s(q+1)$-cycle.\\

\noindent From now on, let $k = \lambda_m(q+1)+r$, for some $m=1, \ldots, s-1$ and some $0\leq r<t_{m+1}(q+1)$.  \\
Our strategy for constructing a $k$-cycle in $\alpha _q$ will be to connect a path on $\Cl_{m+1}$ (note that $m<s$) to $O$ and $\Po_m$. The paths on $\Cl_{m+1}$ we will consider always starts at $P_{m+1,m}$, which will be connected to $P_{m,m} \in \Po_m$ by using $l_m$.
\begin{figure}[htbp]
\begin{center}
\includegraphics[height=1.7in]{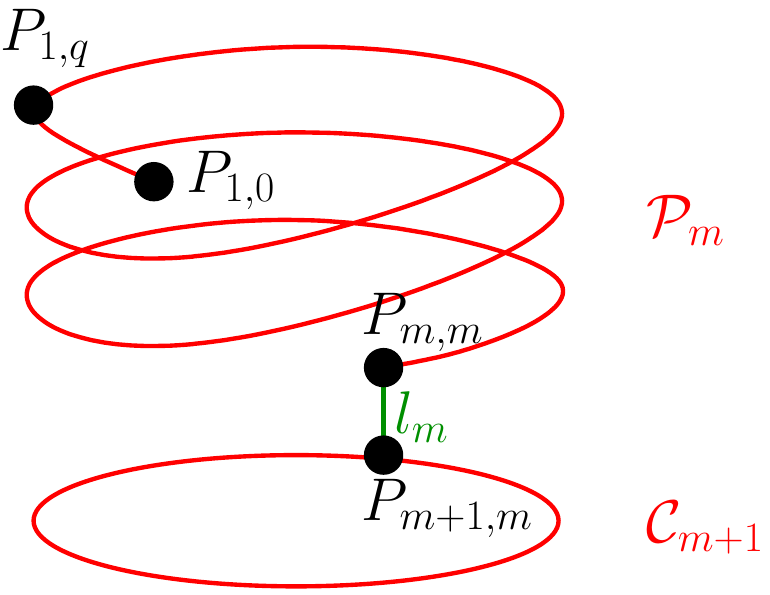}  \\
\caption{Connecting $\Po_m$ and $\mathcal{C}_{m+1}$}
\end{center}
\end{figure}

We consider several cases. \\
\noindent \textbf{(a)} If $r\equiv 3 \pmod{q+1}$ we first get a positive path on $r-1$ vertices on $\Cl_{m+1}$ that ends on a point $Q_{m+1} \in l_{m+1}$.  We then  join $P_{1,0}$ with $O$ using $l_0$, $Q_{m+1}$ with $O$ using $l_{m+1}$. The result is a cycle of the desired length.
%\begin{center}
%\includegraphics[height=1.5in]{Parta.pdf}  \\
%Part a
%\end{center}
\textbf{(b)} If $r\equiv 1 \pmod{q+1}$ we consider a negative $(r-2)$-path on $\Cl_{m+1}$ that ends on a point $Q_{m+1} \in l_{m+1}$. We finish the construction as in case \textbf{(a)}.
%\begin{center}
%\includegraphics[height=1.5in]{Partb.pdf}  \\
%Part b
%\end{center}
\textbf{(c)} If $r\equiv 2 \pmod{q+1}$ or $r\equiv 0 \pmod{q+1}$, then we cut $\mathcal{P}_m$ short one vertex, so it ends at $P_{1,q}$. We get the path in $\Cl_{m+1}$ as in part \textbf{(a)} (for $r\equiv 2 \pmod{q+1}$) or \textbf{(b)} (for $r\equiv 0 \pmod{q+1}$). We close the cycle by joining $P_{1,q}$ with $O$ using $l_q$, $Q_{m+1}$ with $O$ using $l_{m+1}$.\\
\textbf{(d)}   If $r\equiv i \pmod{q+1}$, where $4\leq i \leq q$. We want to get a positive $(r-2)$-path on $\Cl_{m+1}$ starting at $P_{m+1,m}$. This path would end at a point on $l_{m+i-2 \mod q+1}$.   \\
\emph{(i)}  If $i\leq q+2-m$, then $m+i-2 \leq q$, and thus this path on $r-1$ vertices ends at a point $Q_{m+i-2}\in l_{m+i-2}$.   We then get a cycle of the desired length by  joining $P_{1,0}$ with $O$ using $l_0$, $Q_{m+i-2}$ with $O$ using $l_{m+i-2}$.\\
\emph{(ii)}  If $i\geq q+3-m$,  then $m+i-2 \geq q+1$, and thus this path on $r-1$ vertices ends at a point $Q_{t-2}\in l_{t-2}$, where $0\leq t-2 \leq m-3$. We next `shift' this path to make it start at $P_{m+1,m+1}$ instead of $P_{m+1,m}$ and add a vertex to make it a path on $r$ vertices. Now this path ends at $Q_{t}\in l_{t}$, where $2\leq t \leq m-1$.  Since the line $l_t$ is needed to construct $\Po_m$ we will need to modify the construction of $\Po_m$ by connecting the cycles $\Cl_1, \Cl_2,\ldots, \Cl_{m+1}$ in the following way
\[
\Cl_1\xrightarrow{l_1} \Cl_2  \xrightarrow{l_2} \cdots \xrightarrow{l_{t-2}}   \Cl_{t-1}\xrightarrow{l_{t-1}} \Cl_t   \xrightarrow{l_{t+1}} \Cl_{t+1} \xrightarrow{l_{t+2}} \cdots      \xrightarrow{l_{m+1}} \Cl_{m+1}
\]
Note that this can be done for all $2 \leq t \leq m-1$, and that doing this means that $P_{t,t}$ is a `loose' vertex, not used in $\Po_m$ anymore.\\
Now we connect this path to the path on $\Cl_{m+1}$ that ends on $Q_t$. The line $l_t$ is now free, and thus it can be used to close the cycle at $O$. We get the cycle
\[
O \xrightarrow{l_0} \Cl_1 \xrightarrow{l_1} \cdots  \xrightarrow{l_{m}}  \Cl_m \xrightarrow{l_{m+1}} P_{m+1,m+1} \rightarrow \cdots \rightarrow Q_t \xrightarrow{l_t} O
\]
This cycle has length:
\[
\lambda_m(q+1) -1 + r +1 =  \lambda_m(q+1) + r  = k
\]
The `minus one' is because of the loose vertex, the `plus one' is because of $O$.
\end{proof}

%%%%%%%%%%%%%%%%%%%%%%%%%%%%%%%%%%%%%%%%%%%%%%%%%%%%%%%%%%%
\section{Cycles in Projective Planes}

In this section we will study embeddings of cycles in finite projective planes. Let $\pi_q$ denote a projective plane of order $q$. We think about $\pi_q$ as obtained from an affine plane $\alpha _q$ by adding a line, denoted $\ell_{\infty}$, consisting of points $(i)$, for $i=0, \cdots, q$.  Using the notations from the previous section,  each of these points $(i)$  is incident with only the following lines: $\ell_{\infty}$, line $l_i$ of $\alpha _q$,  and the  $q-1$ lines of $\alpha _q$ parallel to $l_i$.
The next statement follows immediately from our work in Section \ref{sec:CyclesAffinePlanes}.

\begin{corollary}\label{lemprojcycles}
Let $\pi_q$ be a projective plane of order $q$. Then, a $k$-cycle can be embedded in $\pi_q$, for all $k=3, \ldots , q^2$.
\end{corollary}

Therefore in order to prove the pancyclicity of $\pi _q$, we need to show that $k$-cycles can be embedded into $\pi_q$ for all $k$, $q^2+1 \leq k \leq q^2+q+1$.  At this point one would expect to use heavily the pancyclicity of $\alpha _q$ for the construction of `long' cycles in $\pi_q$,  but we could not make use of this idea.  Instead,  we base our construction methods on using the cycles $\Cl_i$ in similar ways to that in the proof of Theorem \ref{thmlongcycles}.

Let $W_1$ be any of the vertices of $\Cl_1$ that are on $l_q$,  and let $V_1= (l_1+W_1) \cap l_0$. It follows that $V_1$ is a vertex of $\Cl_1$, and that $l_1+W_1$ is an edge of $\Cl_1$. Similarly, for $2\leq i \leq s$, let $W_i \in l_{i-2}$ be a vertex of $\Cl_i$, and  $V_i= (l_i + W_i) \cap l_{i-1}$. Hence, $V_i$ is a vertex of $\Cl_i$, and  $l_i + W_i$ is an edge of $\Cl_i$. For each $i=1, \cdots , s$, let $[V_i,W_i]$ denote the $V_iW_i$-path in $\Cl_i$, different from the edge $V_iW_i$.  Next we define $U_i$ to be the vertex of $\Cl_i$ that is on $l_q$ and that is the closest to $V_i$, when we move from $V_i$ towards $W_i$ along $\Cl_i$. By $[V_iU_i]$ we denote the subpath of $[V_i,W_i]$ having $q-i+2$ vertices and endpoints $V_i$ and $U_i$.

\begin{figure}[htbp]
\begin{center}
\includegraphics[height=1.65in]{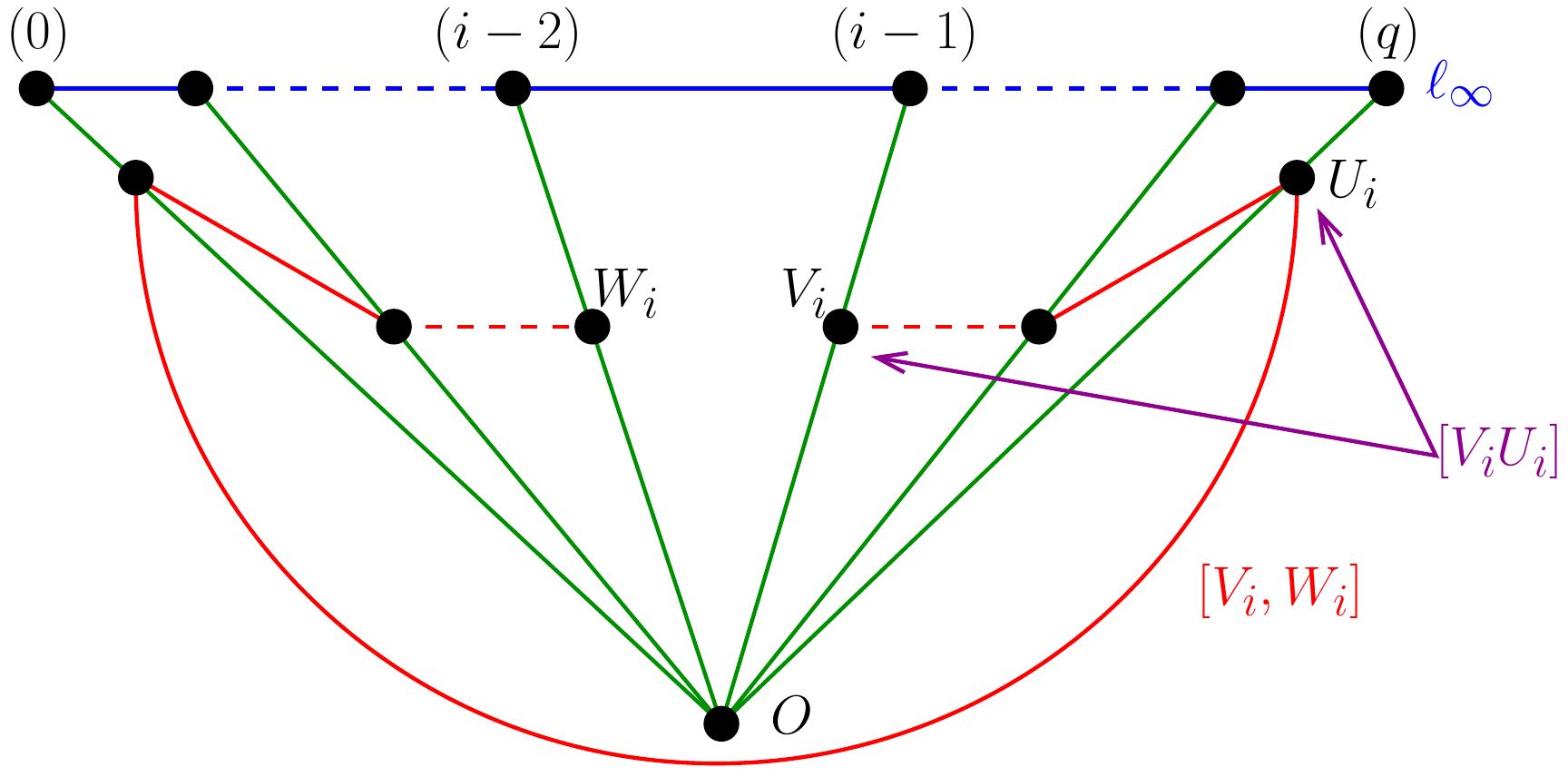} \\
\caption{Paths $[V_i,W_i]$ and $[V_iU_i]$}
\end{center}
\end{figure}

\begin{figure}[htbp]
\begin{center}
\includegraphics[height=1.65in]{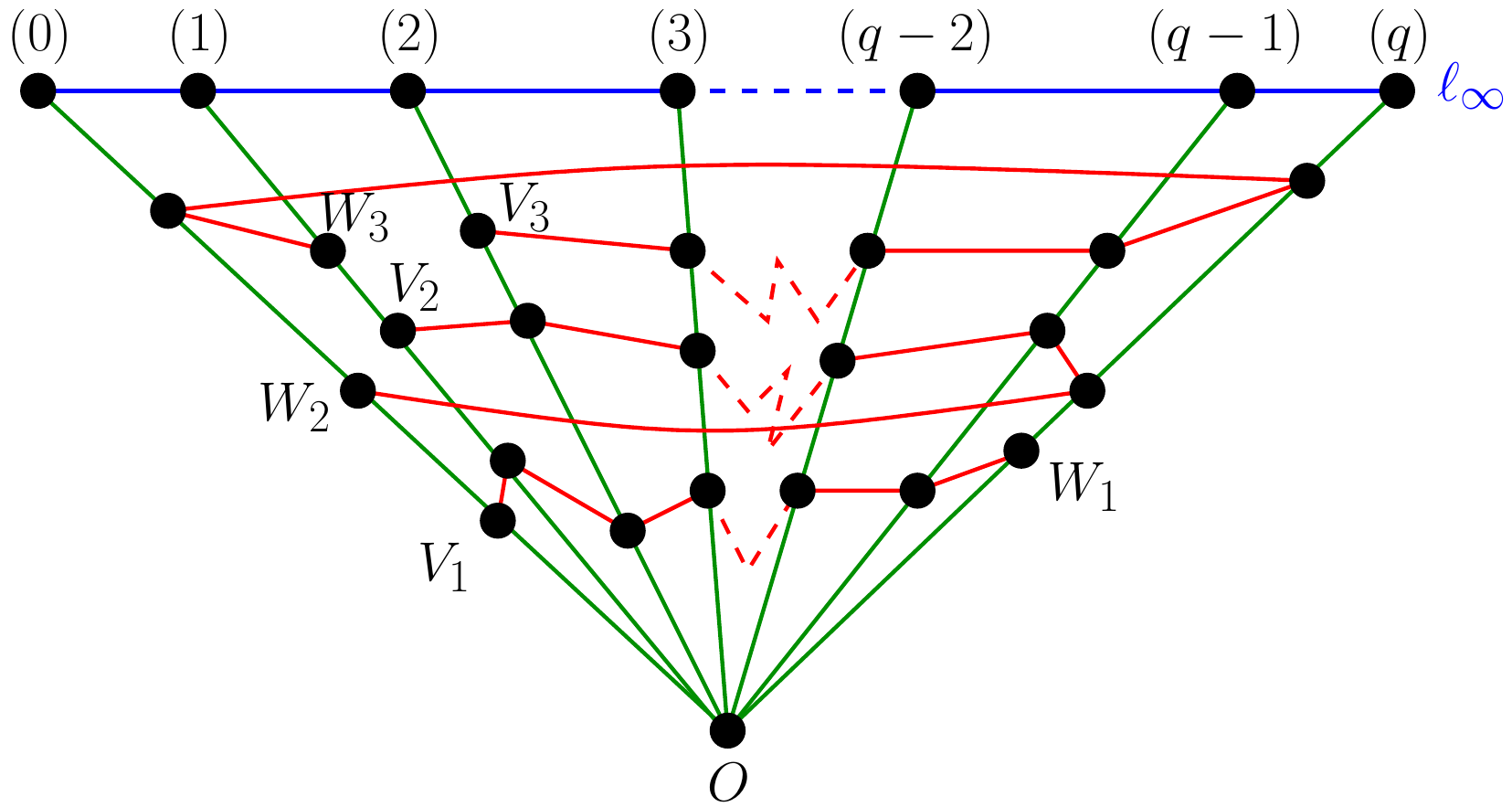} 
\caption{Paths  $[V_1,W_1]$,  $[V_2,W_2]$, and  $[V_3,W_3]$}
\end{center}
\end{figure}

Recall that $(i)= l_i \cap \ell_{\infty}$. We now construct a path $\Po$ (for $s\geq 2$) by connecting $W_i$ with $(i)$ using $l_i+W_i$ (which is not an edge of $[V_i,W_i]$), and connecting $(i)$ with $V_{i+1}$ using $l_i$. Thus $\Po$ is the path:
\[
[V_1,W_1] \xrightarrow{l_1+W_1} (1) \xrightarrow{l_1} [V_2,W_2]  \xrightarrow{l_2+W_2} (2) \rightarrow \cdots \rightarrow (s-1) \xrightarrow{l_{s-1}}  [V_s,W_s].
\]
For $s=1$, $\Po$ is obtained from the cycle $\Cl _1$ by removing the edge $V_1W_1$.

Note that for all $s$, $\Po$  has $(q^2-1)+(s-1)=q^2+s-2$ vertices. The lines $l_{s}, \cdots , l_q, l_0$, $l_s+W_s$, and $\ell_{\infty}$ have not been used in the construction of $\Po$, and neither have the points $(s), \cdots , (q), (0)$, and $O$.

\begin{figure}[htbp]
\begin{center}
\includegraphics[height=1.65in]{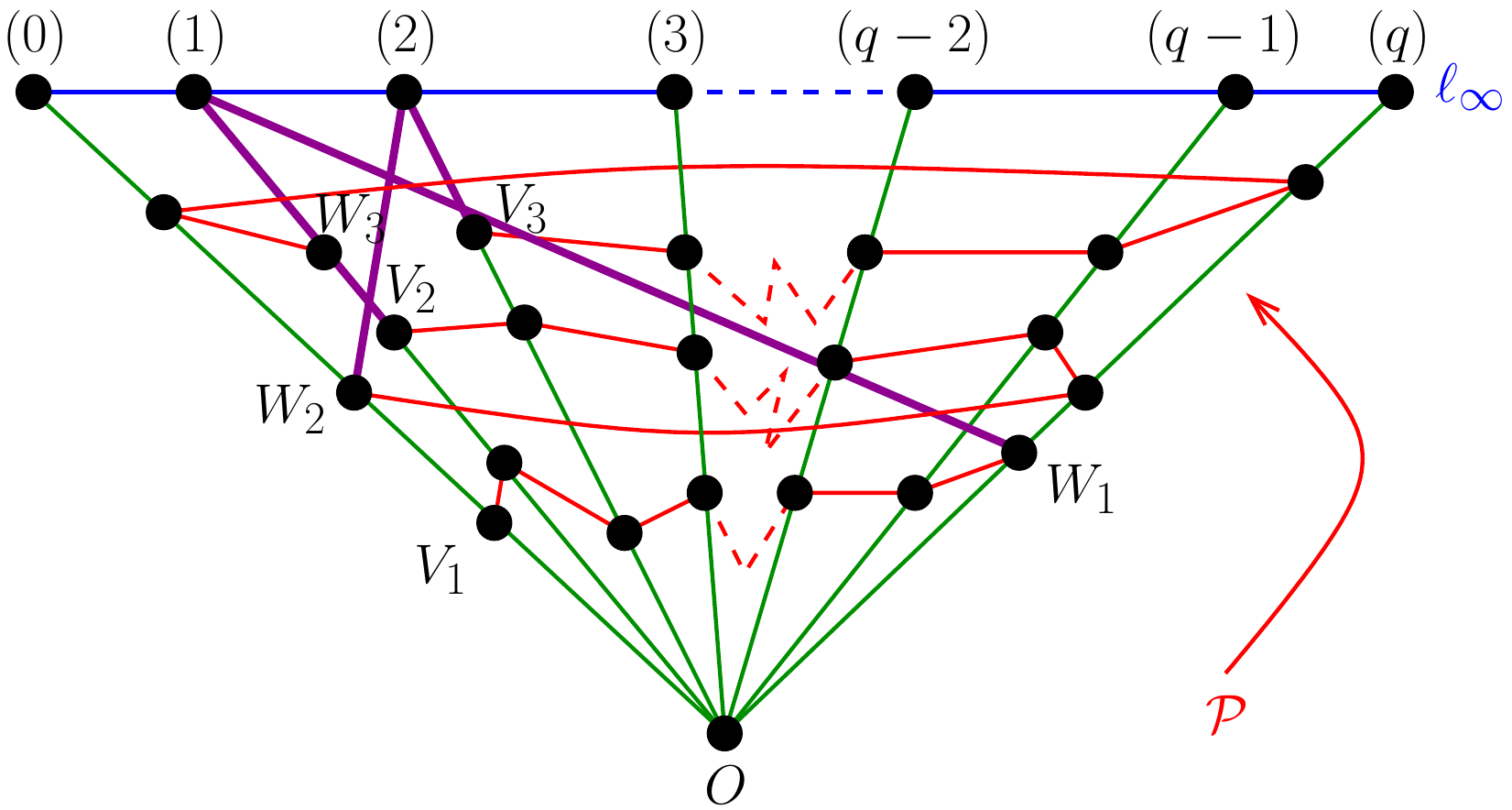} 
\caption{Paths $[V_1,W_1]$,  $[V_2,W_2], [V_3,W_3]$ joined into a path $\Po$. Its endpoints are $V_1$ and $W_3$}
\end{center}
\end{figure}

\begin{figure}[htbp]
\begin{center}
\includegraphics[height=1.65in]{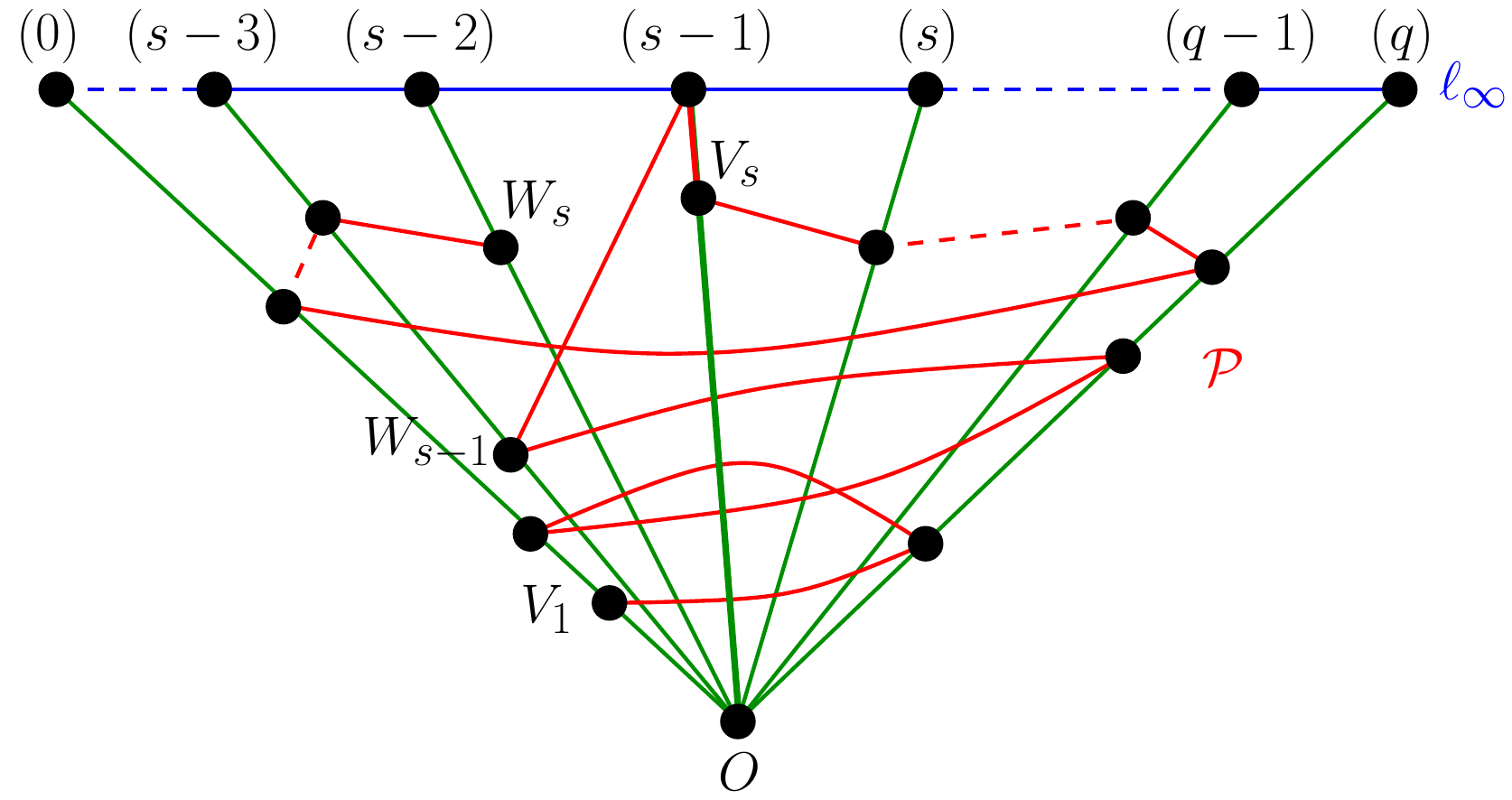} 
\caption{A simple diagram of the path $\Po$}
\end{center}
\end{figure}

Now we begin our construction of $k$-cycles in $\pi_q$ of lengths from $q^2+1$ to $q^2+q+1$ by using the path $\Po$ and/or modifications of it. Recall that $s$ denotes the number of all cycles $\Cl_i$,  or of all paths $[V_i,W_i]$, and that $1\le s\le q-1$.  We will first construct cycles of length between $q^2+1$ to $q^2+s+2$, and then the ones that are longer than $q^2+s+2$.

\begin{lemma}\label{lempathsq^2+1->q^2+s}
Cycles of length ranging from $q^2+1$ to $q^2+s+2$ can be embedded in~$\pi_q$.
\end{lemma}

\begin{proof}
Using the path $\Po$ we can construct
\begin{enumerate}
\item a cycle of length $q^2+s+2$:
\[
\underbrace{V_1 \rightarrow \cdots \rightarrow W_s}_{in \ \Po}  \xrightarrow{l_s+W_s} (s) \xrightarrow{l_s} O \xrightarrow{l_q} (q) \xrightarrow{\ell_{\infty}}  (0) \xrightarrow{l_0} V_1,
\]
\item a cycle of length $q^2+s+1$:
\[
\underbrace{V_1 \rightarrow \cdots \rightarrow W_s}_{in \ \Po}  \xrightarrow{l_s+W_s}  (s) \xrightarrow{\ell_{\infty}} (q) \xrightarrow{l_q} O \xrightarrow{l_0} V_1,\, \text{and}
\]
\item a cycle of length $q^2+s$:
\[
\underbrace{V_1 \rightarrow \cdots \rightarrow W_s}_{in \ \Po} \xrightarrow{l_s+W_s}  (s) \xrightarrow{\ell_{\infty}} (0) \xrightarrow{l_0}  V_1.
\]
Note that the lines $l_{s}, \cdots , l_q$ have not been used in the construction of this last cycle, and neither have the points $(s+1), \cdots , (q)$, and $O$. We will denote this cycle by $\mathcal{C}$.

\bigskip

\end{enumerate}
If  $q+1\leq 2s$, then, for every $q-s+1\leq i \leq s$, let us modify $\mathcal{C}$ in the following way:
\begin{itemize}
\item Delete $q-i+1$ vertices of the path $[V_iU_i]$, all except $U_i$.
\item Connect $(i-1)$ with $O$ (recall that $(i-1)$ was  connected to $V_i$ in $\mathcal{C}$ via $l_{i-1}$).
\item Connect $O$ with $U_i$ using $l_q$.
\end{itemize}
This yields the cycle
\[
\underbrace{U_i \rightarrow \cdots \rightarrow (i-1)}_{in \ \mathcal{C}}  \xrightarrow{l_{i-1}}  O  \xrightarrow{l_q} U_i
\]
which has length $(q^2+s)-(q-i+1)+1 = q^2-q+s+i$.

Since $q-s+1\leq i \leq s$, the length of this cycle ranges between $q^2+1$ and $(q^2+s)-(q-s)$. Note that if $q+1> 2s$, then $(q^2+s)-(q-s) <  q^2  +1$. So, for all relevant values of $s$ we have been able to construct cycles with lengths ranging from $q^2+1$ to $(q^2+s)-(q-s)$.

\medbreak

Next we want to construct $k$-cycles for $(q^2+s)-(q-s)<k<q^2+s$. In order to do that we need to set more notation.

Let us relabel the vertices in the path $[V_sU_s]$ by $V_s =P_{s-1}, P_s, \cdots , P_{q-1}, P_q$, where $P_i\in l_i$, for all $i=s, \ldots , q$.  Note that $P_q = U_s$.  For $s-1\le i< q$, let $[P_iP_j]$ denote the subpath of $[V_sU_s]$ joining $P_i$ and $P_j$.

Note that a cycle of length $(q^2+s)-(q-s)$ vertices may be obtained using $i=s$ in the previous construction. We want to use a similar construction to get a cycle of length $(q^2+s)-(q-s)+1$. We modify $\mathcal{C}$ by replacing its subpath $[P_{s-1}P_{q-1}]$ by a path
\[
(s-1) \xrightarrow{l_{s-1}}  O \xrightarrow{l_{q-1}} P_{q-1}  \rightarrow P_{q}
\]
leading to the following cycle of length $(q^2+s)-(q-s) +1$.
\[
\underbrace{U_s \rightarrow \cdots \rightarrow (s-1)}_{in \ \mathcal{C}} \xrightarrow{l_{s-1}}  O \xrightarrow{l_{q-1}} P_{q-1}  \rightarrow P_{q}
\]
Note that we are using here that $s\leq q-1$.

Now, to create cycles of length larger than $(q^2+s)-(q-s)+1$ we use the following strategy.

For every $i=s, \cdots , q-1$,  we modify  $\mathcal{C}$, by connecting $P_i$ with $O$, and  $O$ with $P_q$ to get the cycle
\[
\underbrace{P_q \rightarrow \cdots \rightarrow P_i}_{in \ \mathcal{C}}  \xrightarrow{l_{i}}  O  \xrightarrow{l_q} P_q ,
\]
which has length $(q^2+s)-(q-i-1)+ 1=(q^2+s)-(q-i -2)$. Since $i=s, \cdots , q-1$, then the length of this cycle ranges from $(q^2+s)-(q-s)+2$ to $(q^2+s)+1$.
\end{proof}

\begin{corollary}
With the same notation used in Lemma \ref{lempathsq^2+1->q^2+s}, if $s=q-1$, then $\pi_q$ is pancyclic.
\end{corollary}

Now we want to construct cycles longer than $q^2+s+2$ for when $1\leq s < q-1$.

\begin{lemma}\label{lempathsq^2+s+2<}
For every $1\leq s < q-1$,  cycles of length ranging from $q^2+s+3$ to $q^2+q+1$ can be embedded in $\pi_q$.
\end{lemma}

\begin{proof}
Just as we did in the proof of Lemma \ref{lempathsq^2+1->q^2+s}, the idea is to modify the path $\Po$ to get the desired cycles. Hence, we will use the same notation introduced earlier in this section, including that used in the proof of Lemma \ref{lempathsq^2+1->q^2+s}.

We first eliminate the edge $l_{s+1}+V_s$ from $\Po$, and connect $W_s$ with $V_s$ using $l_s+W_s$. This gives us a path $\tilde{\Po}$ that has the same length of $\Po$ ($q^2+s-2$ vertices) with endpoints $V_1$ and $P_s$.

\begin{figure}[htbp]
\begin{center}
\includegraphics[height=1.65in]{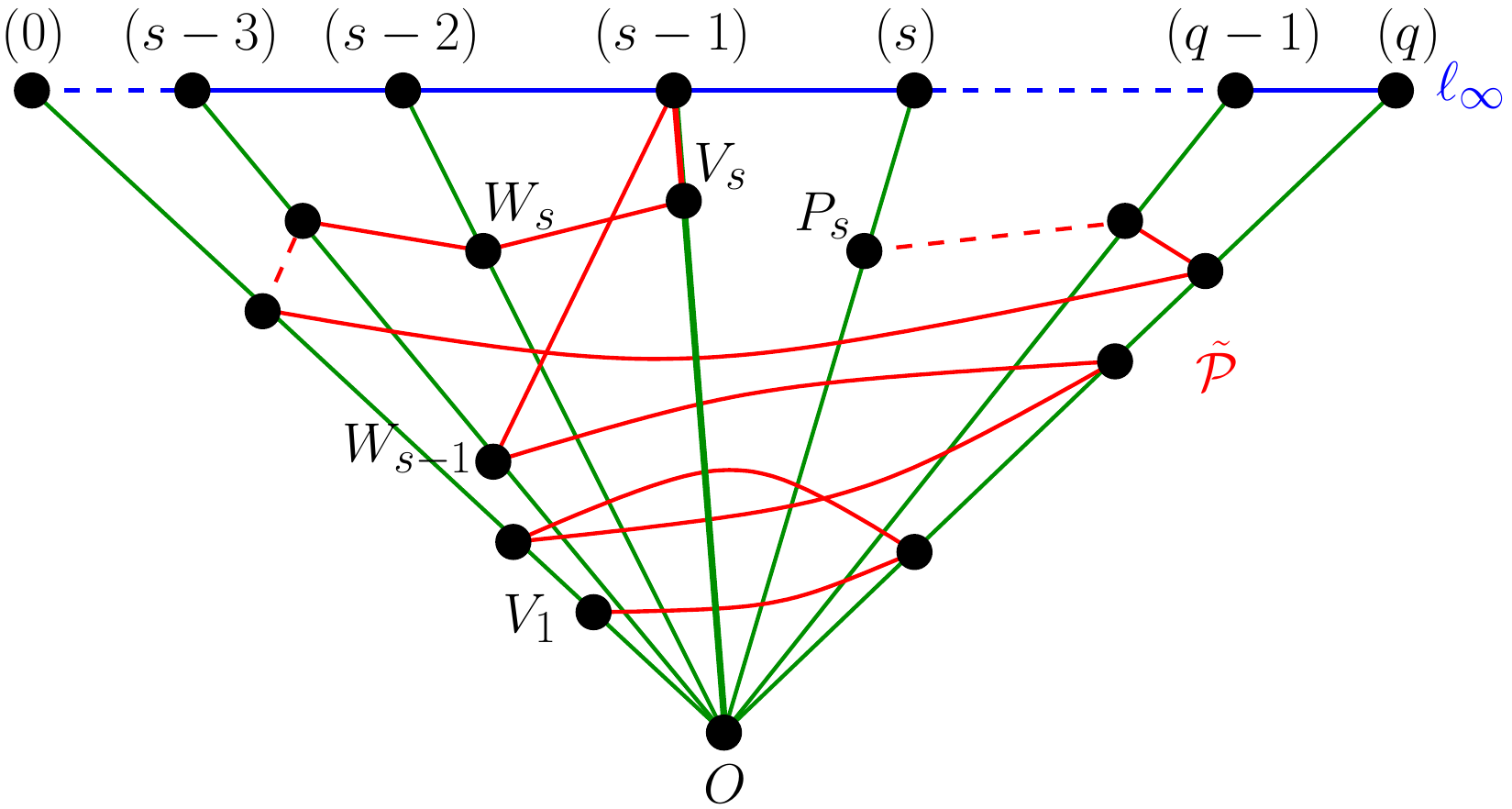} \\
\caption{The path $\tilde{\Po}$}
\end{center}
\end{figure}

\hspace{.2in}

Note that the lines $l_{s}, \cdots , l_q, l_0$, $l_{s+1}+V_s$, and $\ell_{\infty}$ have not been used, neither have the points $(s), \cdots , (q), (0)$, and $O$.

If we now eliminate the edge $l_{s+2}+P_s$ and, instead, connect $P_s$ with $P_{s+1}$ using
\[
P_s \xrightarrow{l_{s+1}+V_s}  (s+1) \xrightarrow{l_{s+1}}   P_{s+1},
\]
we get a path $\G_1$ in $(q^2+s-2)+1$ vertices (one more than $\tilde{P}$). We may close this path into a cycle by
\[
V_1  \xrightarrow{l_{0}} (0)  \xrightarrow{\ell_{\infty}}  (s)  \xrightarrow{l_{s}}  P_s \xrightarrow{l_{s+1}+V_s}  (s+1) \xrightarrow{l_{s+1}}   \underbrace{P_{s+1} \rightarrow \cdots \rightarrow V_1}_{in \ \tilde{\Po}}
\]
which has length $(q^2+s-1)+2=q^2+s+1$.

Now we eliminate  the edge $l_{s+3 \mod q+1}+P_{s+1}$ from $\G_1$ and instead connect $P_{s+1}$ with $P_{s+2}$ using
\[
P_{s+1} \xrightarrow{l_{s+2}+P_s}  (s+2) \xrightarrow{l_{s+2}}   P_{s+2}.
\]
This yields a path $\G_2$ in  $(q^2+s-2)+2$ vertices (two more than $\tilde{\Po}$). We may close this path into a cycle by using $\tilde{\Po}$ as above
{\small
\[
V_1  \xrightarrow{l_{0}} (0)  \xrightarrow{\ell_{\infty}}  (s)  \xrightarrow{l_{s}}  P_s \xrightarrow{l_{s+1}+V_s} (s+1) \xrightarrow{l_{s+1}} P_{s+1} \xrightarrow{l_{s+2}+P_s}  (s+2) \xrightarrow{l_{s+2}}    \underbrace{P_{s+2} \rightarrow \cdots \rightarrow V_1}_{in \ \tilde{\Po}}
\]
}
which has length $(q^2+s)+2=q^2+s+2$.

In general, for $1\leq i< q-s$ (and thus $s+i +1< q+1$), given a path $\G_i$ of length $(q^2+s-2)+i$ constructed as above we can eliminate the edge $l_{s+i+2 \mod q+1}+P_{s+i}$ from $\G_i$ and instead connect $P_{s+i}$ with $P_{s+i+1}$ using
\[
P_{s+i} \xrightarrow{l_{s+i+1}+P_{s+i-1}}  (s+i+1) \xrightarrow{l_{s+i+1}}   P_{s+i+1}
\]
this yields a path $\G_{i+1}$ in  $(q^2+s-2)+i+1$ vertices ($i+1$ more than $\tilde{\Po}$). We may close this path into a cycle as we did above
\[
V_1  \xrightarrow{l_{0}} (0)  \xrightarrow{\ell_{\infty}}  (s)  \xrightarrow{l_{s}}  P_s \xrightarrow{l_{s+1}+V_s}  (s+1) \xrightarrow{l_{s+1}} P_{s+1}\xrightarrow{l_{s+2}+P_s}  (s+2)  \rightarrow \cdots \hspace{1in}
\]
\[
\hspace{1.3in} \cdots \rightarrow  P_{s+i} \xrightarrow{l_{s+i+1}+P_{s+i-1}}  (s+i+1) \xrightarrow{l_{s+i+1}} \underbrace{P_{s+i+1}  \rightarrow \cdots \rightarrow V_1}_{in \ \tilde{\Po}}
\]
which has length $q^2+s+i+1$.

This will yield cycles of length up to $q^2+q$. The line not used in the $(q^2+q)$-cycle $\Q$ is
$l_0+P_{q-1}$, and the point not used is $O$. Figure \ref{fig:Q} gives an idea of what $\Q$ looks like.

\begin{figure}[htbp]
\begin{center}
\includegraphics[height=1.8in]{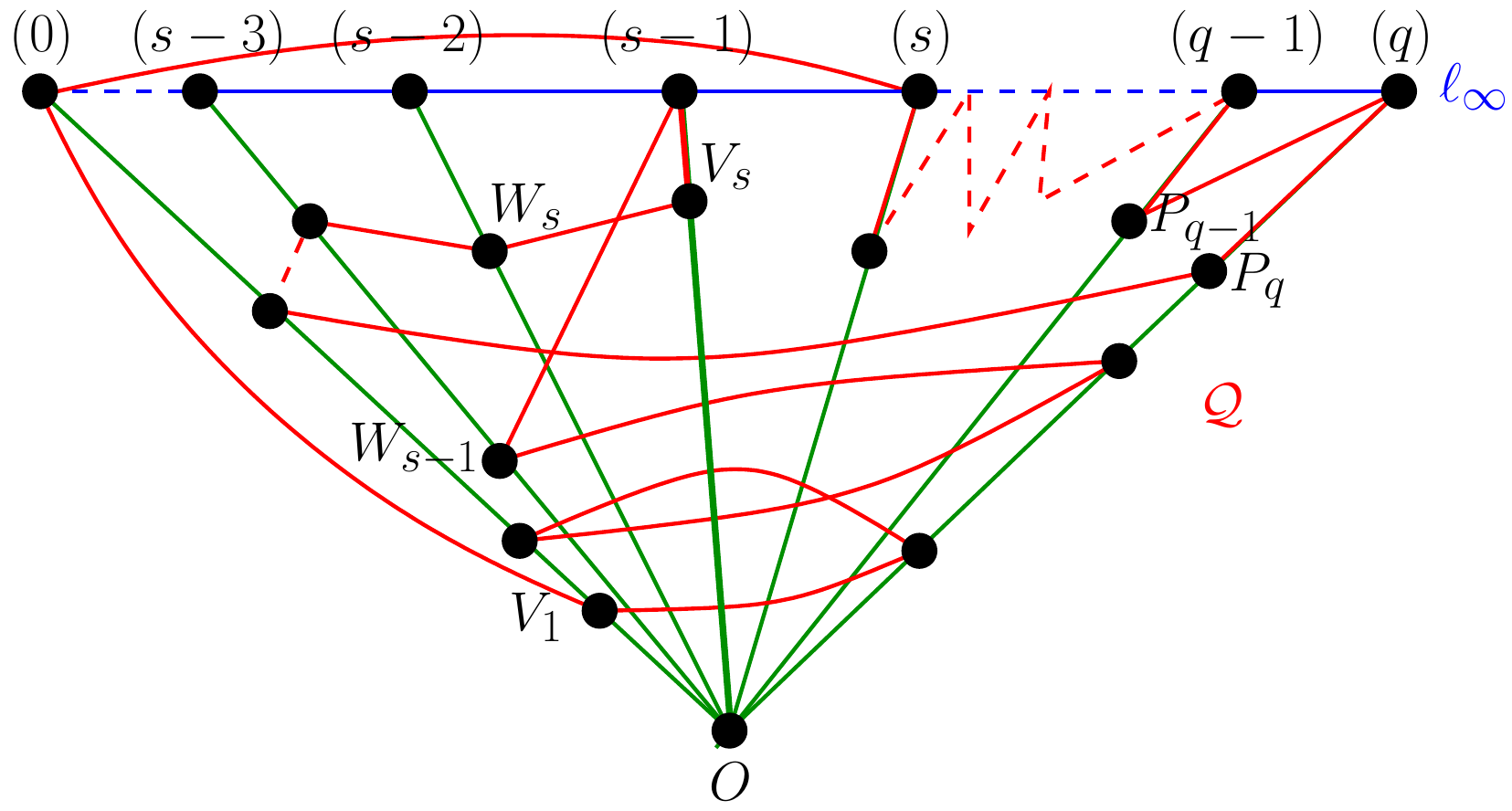} \\
\caption{Cycle $\Q$}\label{fig:Q}
\end{center}
\end{figure}

\hspace{.1in}

In order to construct a $(q^2+q+1)$-cycle we use $\Q$ and modified it as follows.
\begin{itemize}
\item eliminate $\ell_{\infty}$, which connected $l_0$ and $l_s$.
\item eliminate $l_{q-1}$, which connected $(q-1)$ and $P_{q-1}$.
\item eliminate $l_{q}$, which connected $(q)$ and $P_{q}$.
\item connect $(s)$ and $(q)$ using $\ell_{\infty}$.
\item connect $(q-1)$ and $O$ using $l_{q-1}$
\item connect $P_q$ and $O$ using $l_{q}$
\item connect $(q-1)$ and $(0)$ using $l_0+P_{q-1}$
\end{itemize}

We get the following hamiltonian cycle:

\begin{figure}[htbp]
\begin{center}
\includegraphics[height=2in]{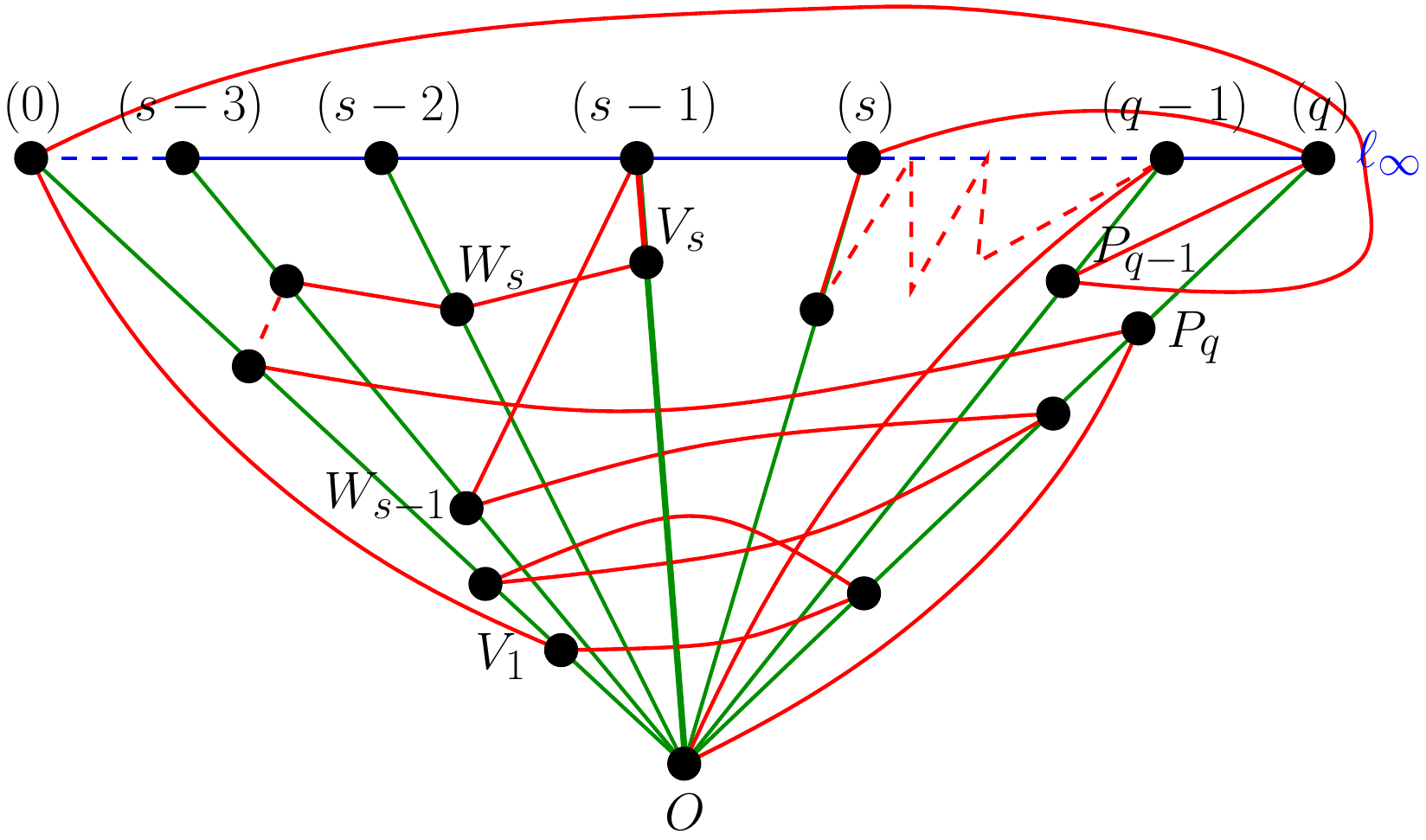} \\
\caption{A hamiltonian cycle}
\end{center}
\end{figure}
\end{proof}

\begin{proof}[Proof of Theorem  \ref{main}]
It follows from Lemmas \ref{lempathsq^2+1->q^2+s} and \ref{lempathsq^2+s+2<}.
\end{proof}

We wish to conclude this paper with a conjecture.   Let $s\ge 1$ and $n\ge 2$.   A finite partial plane
$\mathcal{G} = ({\Po}, {\Li} ; \I)$ is called a {\it generalized $n$-gon of order $s$} if its Levi graph is
$(s+1)$-regular,  has diameter  $n$, and has girth  $2n$.  It is known that a generalized $n$-gons of order
$s$ exists only for $n=2,3,4,6$, see Feit and Higman \cite{FH}.  It is easy to argue that the number of points
and the number of lines in the generalized $n$-gon is $p_s^{(n)} : = s^{n-1} + s^{n-2} + \cdots + s+1$ .  Note
that a projective plane of order $q$ is a generalized $3$-gon (generalized triangle) of order $q$, and so
$p_q^{(3)}  =  q^2 + q +1 = n_q$ -- the notation used in this paper earlier.

\begin{conjecture}
Let $s\ge 2$ and $n\ge 3$. Then $C_k \hookrightarrow \mathcal{G}$ for all $k$, $n\le k\le p_s^{(n)}$.
\end{conjecture}

\hspace{.2in}

\noindent \textbf{Acknowledgement:} The authors are thankful to Benny Sudakov for the clarification of  related results obtained by probabilistic methods.

\bigskip

\bigskip

%%%%%%%%%%%%%%%%%%%%%%%%%%%%%%%%%%%%%%%%%%%%%%%%%%%%%%%%%%%%

\end{document}